\documentclass{amsart}
\usepackage{amsmath,amssymb,url}
\newtheorem{theorem}{Theorem}[section]
\newtheorem{lemma}[theorem]{Lemma}
\newtheorem{proposition}[theorem]{Proposition}

\newtheorem{corollary}[theorem]{Corollary}
\newtheorem{conj}[theorem]{Conjecture}
\newcommand{\Cay}{\mathrm{Cay}}
\newcommand{\C}{\mathrm{C}}
\newcommand{\V}{\mathrm{V}}
\newcommand{\Fix}{\mathop{\textrm{Fix}}}
\newcommand{\E}{\mathrm{E}}
\newcommand{\SG}{\mathrm{SG}}
\newcommand{\im}{\mathrm{im}}
\newcommand{\rk}{\mathrm{dim}}
\newcommand{\rr}{\mathrm{rk}}
\newcommand{\Aut}{\mathrm{Aut}}
\newcommand{\Sym}{\mathrm{Sym}}
\newcommand{\NM}{\mathcal{N}_m}
\newcommand{\GL}{\mathrm{GL}}
\def\nor#1#2{{\bf N}_{#1}(#2)}
\def\cen#1#2{{\bf C}_{#1}(#2)}

\numberwithin{equation}{section}

\title[Enumerating vertex-transitive graphs]{Asymptotic enumeration of vertex-transitive graphs of fixed valency}

\author[P. Poto\v{c}nik]{Primo\v{z} Poto\v{c}nik}
\address{Primo\v{z} Poto\v{c}nik, Institute of Mathematics, Physics, and
  Mechanics, \newline
Jadranska 19, 1000 Ljubljana, Slovenia}\email{primoz.potocnik@fmf.uni-lj.si}

\author[P. Spiga]{Pablo Spiga}
\address{Pablo Spiga, Dipartimento di Matematica Pura e Applicata,\newline
University of Milano-Bicocca,
Via Cozzi~53, 20126 Milano Italy}  \email{pablo.spiga@unimib.it}

\author[G. Verret]{Gabriel Verret}
\address{Gabriel Verret, Faculty of Mathematics, Natural Sciences and Info. Tech.,  \newline
University of Primorska, Glagolja\v{s}ka 8, 6000 Koper, Slovenia}
\email{gabriel.verret@fmf.uni-lj.si}


\begin{document}
\begin{abstract}
Let $G$ be a group and let $S$ be an inverse-closed and identity-free generating set of $G$. The \emph{Cayley graph} $\Cay(G,S)$ has vertex-set $G$ and two vertices $u$ and $v$ are adjacent if and only if  $uv^{-1}\in S$. Let $CAY_d(n)$ be the number of  isomorphism classes of $d$-valent Cayley graphs of order at most $n$. We show that $\log( CAY_d(n))\in\Theta (d(\log n)^2)$, as $n\to\infty$. We also obtain some stronger results in the case $d=3$.
\end{abstract}
\maketitle

\section{Introduction}
Unless explicitly stated otherwise, all graphs considered in this paper are finite, connected and simple (undirected, loopless and with no multiple edges). Given a graph $\Gamma$, its vertex-set is denoted by $\V(\Gamma)$, and $|\V(\Gamma)|$ is called the \emph{order} of $\Gamma$. The automorphism group of $\Gamma$ is denoted by $\Aut(\Gamma)$, and $\Gamma$ is said to be $G$-\emph{vertex-transitive} if $G$ is a subgroup of $\Aut(\Gamma)$ acting transitively on the vertex-set of $\Gamma$.  When $G=\Aut(\Gamma)$, the prefix $G$ in the above notation is sometimes omitted.

Clearly, a vertex-transitive graph is regular, that is, all its vertices have the same valency. We are interested in the problem of enumerating vertex-transitive graphs of a fixed valency $d\geq 3$. Let $vt_d(n)$ denote the number of isomorphism classes of $d$-valent vertex-transitive graphs of order $n$. Note that $vt_d(n)$ behaves very irregularly as a function of $n$. For example, it can be seen that $vt_d(n)$ is relatively small if $n$ is a prime, while $vt_d(n)$ is relatively large if $n$ is divisible by a large power of a prime. Finding an exact formula for $vt_d(n)$ seems out of reach anyway.

To avoid both of these issues with $vt_d(n)$, we consider instead the asymptotic behaviour of its summatory function: let $VT_d(n)$ denote the number of isomorphism classes of $d$-valent vertex-transitive graphs of order at most $n$. Our goal is then to understand the asymptotic behaviour of $VT_d(n)$ as $n\to\infty$.

We will also be interested in enumerating some special families of vertex-transitive graphs. A graph is called a $\emph{Cayley graph}$ if it admits a group of automorphisms acting regularly on its vertices. A graph is called a \emph{graphical regular representation} or \emph{GRR} if its full automorphism group acts regularly on its vertices. Clearly, a Cayley graph is necessarily vertex-transitive and a GRR is necessarily a Cayley graph.
Let $CAY_d(n)$ (respectively $GRR_d(n)$) denote the number of  isomorphism classes of $d$-valent Cayley graphs (respectively GRR) of order at most $n$. The first main result of our paper is the following.

\begin{theorem}\label{Main1}
There exist positive constants $a$, $b$ and $c$ such that, for every $d\geq 3$,
$$n^{a d\log n}-c\leq CAY_d(n)\leq n^{bd\log n}.$$
\end{theorem}

(Note that, unless otherwise stated, the logarithm is the natural one.) Since Cayley graphs are vertex-transitive, it follows from Theorem~\ref{Main1} that $n^{a d\log n}-c\leq VT_d(n)$. In the case $d=3$, we were able to obtain stronger results.

\begin{theorem}\label{Main2}
There exist positive constants $a$, $b$ and $c$ such that
$$n^{a\log n}-c\leq GRR_3(n)\leq CAY_3(n)\leq VT_3(n)\leq n^{b\log n}.$$
\end{theorem}

We conjecture that Theorem~\ref{Main2} also holds when $3$ is replaced by a larger integer.

\begin{conj}\label{Conj3}
There exist positive constants $a$, $b$ and $c$ such that, for every $d\geq 3$,
$$n^{ad\log n}-c\leq GRR_d(n)\leq CAY_d(n)\leq VT_d(n)\leq n^{bd\log n}.$$
\end{conj}

The techniques used in this paper can be applied to prove a variety of similar results. We give Theorem~\ref{theo:5arc} simply as an illustration. An \emph{$s$-arc} in a graph $\Gamma$ is a sequence of $s+1$ vertices such that every two consecutive vertices are adjacent and every three consecutive vertices are pairwise distinct. If $\Gamma$ is a graph such that $\Aut(\Gamma)$ acts transitively on the $s$-arcs of $\Gamma$, then $\Gamma$ is called \emph{$s$-arc-transitive}.

\begin{theorem}\label{theo:5arc}
There exists positive constants $a$ and $b$ such that the number of isomorphism classes of $3$-valent $5$-arc-transitive graphs of order at most $n$ is at least $n^{a\log n}-b$.
\end{theorem}

We also propose two more conjectures. It has been conjectured that almost all vertex-transitive graphs are Cayley graphs (see~\cite{McKayPraeger} for example). We conjecture that this is true for fixed valency.

\begin{conj}\label{Conj1}
Let $d\geq 3$. Then
$$\frac{CAY_d(n)}{VT_d(n)}\to 1, \phantom{yo!} (n\to\infty).$$
\end{conj}

Let $G$ be a group and let $S$ be an inverse-closed and identity-free generating set of $G$. Let $\Cay(G,S)$ be the graph with vertex-set $G$ and two vertices $u$ and $v$ are adjacent if and only if  $uv^{-1}\in S$. It is easy to see that $\Cay(G,S)$ is a Cayley graph and that every Cayley graph can be represented in this way (not necessarily in a unique way). 

It is commonly believed that almost all Cayley graphs are GRRs. Some care must be used in interpreting this statement. Some authors consider, as $G$ runs through the groups of order $n$, the proportion of generating sets $S$ of $G$ with $\Cay(G,S)$ a GRR, and then let $n\to\infty$ (see~\cite{BabaiAndMoron} for example). We prefer to consider the proportion of GRRs  amongst all Cayley graphs of order at most $n$, as $n\to\infty$. (Since a given Cayley graph may be realised by several different pairs $(G,S)$, these two interpretations are not necessarily equivalent.)  More precisely, we conjecture the following.

\begin{conj}\label{Conj2}
Let $d\geq 3$. Then
$$\frac{GRR_d(n)}{CAY_d(n)}\to 1, \phantom{yo!} (n\to\infty).$$
\end{conj}

Note that Conjecture~\ref{Conj3} would follow from Theorem~\ref{Main1} together with Conjectures~\ref{Conj1} and~\ref{Conj2}.

We now give a brief outline of the rest of the paper. In Section~\ref{sec:upper}, we obtain the upper bound in Theorem~\ref{Main1} as an easy consequence of some results in the field of normal subgroup growth in finitely generated groups. In Section~\ref{sec:upper2}, we apply one of our recent results on the growth of automorphism groups of $3$-valent vertex-transitive graphs to obtain the upper bound in Theorem~\ref{Main2}.

As a consequence of the work in this paper, we also obtain some bounds on the number of $2$-groups of a given order generated by a given number of involutions. This and related problems are discussed in Section~\ref{sec:2groups}. Section~\ref{sec:main} is the longest. This is where the lower bounds in Theorems~\ref{Main1} and~\ref{Main2} are proved. Finally, Theorem~\ref{theo:5arc} is proved in Section~\ref{sec:5AT}.

\section{The upper bound in Theorem~\ref{Main1}}\label{sec:upper}
In this section, we obtain the upper bound in Theorem~\ref{Main1} as an easy consequence of the following theorem due to Lubotzky.

\begin{theorem}[{{\cite[Theorem~1]{Alex}}}]\label{tm1}  There exists a positive constant $a$ such that the number of isomorphism classes of groups which are $d$-generated and of order at most $n$ is at most $n^{ad\log n}$.
\end{theorem}

(Note that we call a group $G$ \emph{$d$-generated} if it can be generated by $d$ elements, and that $d$ is not necessarily the minimum number of generators of $G$.) Some information about the constant $a$ in Theorem~\ref{tm1} can be found in~\cite[Section~$3$, Remark~$1$]{Alex}.

\begin{proof}[Proof of the upper bound in Theorem~\ref{Main1}]
Let $\Gamma=\Cay(G,S)$ be a $d$-valent Cayley graph of order at most $n$. Since $\Gamma$ is connected, $G$ is $d$-generated. By Theorem~\ref{tm1}, there exists a constant $a$ such that the number of possible isomorphism classes for $G$ is at most $n^{a d\log n}$. For a fixed $G$, there are at most $|G|^{|S|}\leq n^d$ possible choices for $S$ and hence at most $n^{d(1+a \log n)}$ choices for $\Gamma$. This shows that $CAY_d(n)\leq n^{d(1+a \log n)}$. Let $b=a+1/\log 4$ and note that $n\geq 4$ and hence $(1+a\log n)\leq b\log n$, and the result follows.
\end{proof}

\section{The upper bound in Theorem~\ref{Main2}}\label{sec:upper2}

In this section, we prove the upper bound in Theorem~\ref{Main2}. 
Our approach is based on Theorem~\ref{cubiclost} below, the proof of which relies on
a classification of cubic vertex-transitive graphs with ``large'' vertex-stabilisers proved in
 \cite{lost}. As will shall see in the proof, the bound $cn^2$ given in the theorem below is rather crude and could  easily be improved if needed.

\begin{theorem}\label{cubiclost}
There exists a positive constant $c$ such that the number of $3$-valent vertex-transitive graphs $\Gamma$ of order at most $n$ with $|\Aut(\Gamma)|> n^2$ is at most $cn^2$.
\end{theorem}
\begin{proof}
Let $n$ be a positive integer. We would like to obtain an upper bound on the number of $3$-valent vertex-transitive graphs $\Gamma$
such that
$$
|\V(\Gamma)| \le n\>\> \hbox{ and } \>\> |\Aut(\Gamma)|> n^2.
\eqno{(*)}
$$
Let $G=\Aut(\Gamma)$ and let $G_v$ denote the stabiliser in $G$ of the vertex $v$ of $\Gamma$. In view of \cite[Corollary 4]{lost},  $\Gamma$  satisfies one of the following:
\begin{itemize}
\item[{\rm (A)}] $\Gamma$ is one of $19$ exceptional graphs (appearing in Tables 1 and 2 of \cite{lost});
\item[{\rm (A')}] $\Gamma$ is a graph (denoted $\SG(\C(r,s))$ in \cite{lost}), 
                            which is uniquely determined by a pair of integers $(r,s)$, satisfying $r \ge 3$, $1 \le s \le r$,
                            and the order of which is $r 2^s$;
\item[{\rm (B)}] $\Gamma$ is arc-transitive and $|G_v| \le 48$;
\item[{\rm (C)}] $|\V(\Gamma)| \ge 8|G_v| \log_2 |G_v|$.
\end{itemize}

Let us now obtain an upper bound on the number of possibilities for $\Gamma$ in each of these cases.

First, case (A) gives rise to at most $19$ more possibilities for  $\Gamma$. In case (A') (that is, in the case when $\Gamma \cong \SG(\C(r,s))$ for some $r$ and $s$), condition $(*)$ implies that
$r 2^s \le n$.  The number of pairs $(r,s)$ satisfying $r 2^s \le n$, and thus the number of possibilities for $\Gamma$ in case (A'),
 is less than $n \log_2 n$.

In case (B) we see that $48n\geq |\Aut(\Gamma)| \ge n^2$ and thus $n\le 48$. 
It follows that there are only finitely many  possibilities for $\Gamma$ in this case. In fact, a recently obtained census of small $3$-valent vertex-transitive graphs \cite{census} reveals that there are precisely $12$ connected $3$-valent vertex-transitive graphs $\Gamma$ of order at most $48$ satisfying $|\Aut(\Gamma)| > |\V(\Gamma)|^2$, implying that case (B) contributes at most $12$ new possibilities for $\Gamma$.

Finally, in case (C) we have $|G_v| <  |\V(\Gamma)| \le n$, and thus  $|\Aut(\Gamma)| < n |\V(\Gamma)| < n^2$, contradicting condition $(*)$. To summarise, there are less than $12+19 + n\log_2 n$ possibilities for the graph $\Gamma$, and the result follows.
\end{proof}

We also need the following elementary observation.

\begin{lemma}\label{3gen} Let $\Gamma$ be a $d$-valent $G$-vertex-transitive graph. Then there exists $H\leq G$ such that $H$ is $d$-generated and $\Gamma$ is $H$-vertex-transitive.
\end{lemma}
\begin{proof}
Let $v$ be a vertex of $\Gamma$ and let $\{v_1,\ldots,v_d\}$ be the neighbourhood of $v$. For every $i\in\{1,\ldots,d\}$, let $g_i$ be an element of $G$ mapping $v$ to $v_i$. Let $H=\langle g_1,\ldots,g_d\rangle$. For every neighbour $v_i$ of $v$, $H$ contains an element mapping $v$ to $v_i$. It follows that, for every $u$ in the $H$-orbit of $v$ and every neighbour $w$ of $u$, $H$ contains an element mapping $u$ to $w$. Since $\Gamma$ is connected, it follows that $H$ is transitive on the vertices of $\Gamma$. 
\end{proof}

\begin{proof}[Proof of the upper bound in Theorem~\ref{Main2}]
Let $\Gamma$ be a $3$-valent vertex-transitive graph of order at most  $n$ and let $A=\Aut(\Gamma)$. First, suppose that $|A|\leq n^2$. By  Lemma~\ref{3gen}, there exists $H\leq A$ such that $H$ is $3$-generated  and $\Gamma$ is $H$-vertex-transitive. As $|H|\leq|A|\leq n^2$, it follows from  Theorem~\ref{tm1} that there exists a positive constant $a$ such that the  number of isomorphism classes for $H$ (as an abstract group) is at most  $(n^2)^{a\log n^2}=n^{4a\log n}$.

Let us now count the number of possible transitive permutation representations of $H$ of degree at most $n$. Each such representation arises from a subgroup $B$ of index at most $n$ in $H$. Since $|H|\leq n^2$, $B$ has order at most $n^2$ and hence is  $\lceil\log_2 n^2\rceil$-generated. As $|H|\leq n^2$, we see that the number of  possible choices for $B$ is at most $(n^2)^{\log n^2}=n^{4\log n}$. This gives  that the number of possible transitive permutation representations of $H$ of degree at most $n$ is at most $n^{4\log n}$.

Finally, given a transitive permutation representation of degree at most $n$ of $H$, we show that the number of possible  $H$-vertex-transitive graphs $\Gamma$  is at most $n^3$. Indeed,  $\Gamma$ is uniquely determined by its edge-set $\E(\Gamma)$ and,  as $H$ is transitive on vertices, $\E(\Gamma)$ is uniquely determined by the  neighbourhood $\{v_1,v_2,v_3\}$ of a given vertex. Clearly, the number of possible choices for  $\{v_1,v_2,v_3\}$ is at most $n^3$. Summing up, it follows that the number of isomorphism classes for $\Gamma$ in this case is at most $n^{4a\log n}\cdot n^{4\log n}\cdot  n^3$.

Now, suppose that $|A|> n^2$. By Theorem~\ref{cubiclost}, there  exists a positive constant $c$ such that the number of isomorphism  classes for $\Gamma$ is at most $cn^2$.  Clearly, there exists a  positive constant $b$ such that $n^{b\log n}\geq cn^2+n^{4a\log n+4\log  n+3}$ for all $n$, and the result follows.
\end{proof}

It would be very interesting to generalise the upper bound in Theorem~\ref{Main2} to valencies higher than $3$. To use our approach, it would first be necessary to obtain a suitable generalisation of Theorem~\ref{cubiclost}.

Observe that in our proof of the upper bound for Theorem~\ref{Main2}, it is very important that the valency $d$ is bounded and does not grow with the number of vertices $n$ of the graph. In comparison, we recall that the automorphism group of the $d$-dimensional cube admits at least $2^{ad^2}$ pairwise non-isomorphic regular subgroups for some positive constant $a$~\cite{pablo}. This would inevitably compromise our approach for generalizing the proof to more general situations.

\section{Enumerating $2$-groups generated by $d$ involutions}\label{sec:2groups}

The problem of estimating the number $f_d(n)$ of isomorphism classes of $d$-generated groups of order $n$ has a rich history. In $1969$, Neumann~\cite{Neumann} considered the function $f_d(n)$ in the case that $n=p^m$ for a fixed prime $p$ and some $m\geq 1$. McIver and Neumann~\cite{IN} showed that $f_d(p^m)\leq p^{\frac{1}{2}(d+1)m^2+O(m)}$. In light of this result, Pyber~\cite{Pyber} conjectured that there exists a function $c$ such that $f_d(n)\leq n^{c(d)\log n}$. Finally, in 2001, Lubotzky~\cite{Alex} proved that $f_d(n)\leq n^{cd\log n}$ for some positive constant $c$ (using some ideas of Mann~\cite{Mann}). No particular effort has been made to determine the best constant $c$, however it follows from~\cite[Section~$3$(1)]{Alex} that $f_d(n)\leq n^{2(d+1)\lambda(n)}$, where $\lambda(n)$ is the number of prime-power divisors of $n$ (that is, if $n=p_1^{e_1}\cdots p_\ell^{e_\ell}$ with $p_1,\ldots,p_\ell$ primes, then $\lambda(n)=e_1+\cdots +e_\ell$). Moreover, Jaikin-Zapirain has tightened both the upper and the lower bounds on $f_d(n)$ in the case when $n=p^m$ for a fixed prime $p$. He has proved~\cite[Theorem~$1.4$]{Andrei} that

\begin{equation}\label{Andreip}
p^{\frac{1}{4}(d-1)m^2+o(m^2)}\leq f_d(n)\leq p^{\frac{1}{2}(d-1)m^2+o(m^2)}.
\end{equation}

In this paper, we will obtain some results on a variant of this problem. Let $g_d(m)$ be the number of isomorphism classes of groups of order $2^m$ which admit a generating set consisting of $d$ involutions.

\begin{theorem}\label{BoundsBounds}
For every $d\geq 3$, we have $2^{\frac{(d-2)^2}{8d}m^2+o(m^2)}\leq g_d(m)\leq 2^{\frac{1}{2}(d-2)m^2+o(m^2)}$.
\end{theorem}

We first prove the upper bound, which easily follows from~\cite{Andrei}.

\begin{proof}[Proof of the upper bound in Theorem~\ref{BoundsBounds}]
Let $d\geq 3$ and let $G$ be a group of order $2^m$ which is generated by $d$ involutions $x_1,\ldots,x_d$. Let $N=\langle x_1x_d,\ldots,x_{d-1}x_d\rangle$. Note that $x_d$ acts by conjugation as inversion on the elements of $\{x_1x_d,\ldots,x_{d-1}x_d\}$, and that $N$ has order at most $2^m$ and is $(d-1)$-generated. Moreover, the isomorphism class of $G$ is uniquely determined by $N$ together with a distinguished set $S$ of $d-1$ generators of $N$. It follows from Eq.~(\ref{Andreip}) that there are at most $2^{\frac{1}{2}(d-2)m^2+o(m^2)}$ possible choices for $N$.  There are at most $|N|^{|S|}\leq (2^m)^{d-1}$ possible choices for $S$, therefore
$$g_d(m)\leq 2^{\frac{1}{2}(d-2)m^2+o(m^2)}2^{m(d-1)}\leq 2^{\frac{1}{2}(d-2)m^2+o(m^2)}.$$
\end{proof}

We did not succeed in adapting the proof of the lower bound in~\cite{Andrei} to a proof of the lower bound in Theorem~\ref{BoundsBounds}: at a critical juncture, Jaikin-Zapirain considers $p$-groups generated by elements of ``large'' order while in our context, the groups need to be generated by involutions. The lower bound in Theorem~\ref{BoundsBounds} will be obtained by slightly different methods in the next section (see Theorem~\ref{thm2}).

\section{The proofs of the lower bounds}\label{sec:main}

The goal of this section is prove the lower bounds in Theorems~\ref{Main1},~\ref{Main2} and~\ref{BoundsBounds}. The proofs of the lower bound on $f_d(p^m)$ in~\cite{Ne} and in~\cite{Andrei} consist in a careful analysis and counting of the normal subgroups of index $p^m$ in the free group on $d$ generators. Our method is similar and based on the same techniques but applied to the universal group generated by $d$ involutions. We first set some notation. Throughout this section,  let $d\geq 3$ and let

\begin{align*}
&&W_d&=\langle x_1,\ldots,x_d\mid x_1^2,\ldots ,x_d^2\rangle,&\\
&&y_1&=x_1x_d,\,\, y_2=x_2x_d,\,\, \ldots,\,\, y_{d-1}=x_{d-1}x_d,&\\
&&F_{d-1}&=\langle y_1,\ldots,y_{d-1}\rangle.&
\end{align*}
 Clearly, $F_{d-1}$ is a free group of rank $d-1$. Moreover, $|W_d:F_{d-1}|=2$ and $W_d=F_{d-1}\rtimes \langle x_d\rangle$ with $y_i^{x_d}=y_i^{-1}$ for every $i\in \{1,\ldots,d-1\}$.

Let $\Sym(d)$ be the symmetric group on $\{1,\ldots,d\}$. For a generator $x_i$ of $W_d$ and a permutation $\sigma\in \Sym(d)$, let $x_i^{\sigma}=x_{i^\sigma}$. This induces a faithful action of $\Sym(d)$ on $W_d$ as a group of automorphism and, in the rest of this section, we consider $\Sym(d)$ as a subgroup of $\Aut(W_d)$.

Given a group $G$ and two subgroups $H,K\leq G$, we write $G^2=\langle g^2\mid g\in G\rangle$ and $[H,K]=\langle[h,k]\mid h\in H, k\in K\rangle$, where $[h,k]=h^{-1}k^{-1}hk$. If $G$ is an elementary abelian $p$-group, then its dimension when viewed as a vector space over the prime field $\mathbb{F}_p$ will be denoted by $\rk(G)$. For a subgroup $H$ of $W_d$, $H/([H,W_d]H^2)$ is an elementary abelian $2$-group. We will denote by $\rr(H)$ the dimension of $H/([H,W_d]H^2)$, in other words,
$$\rr(H)=\rk(H/([H,W_d]H^2)).$$

\subsection{The lower central series of $W_d$}\label{low}
For a group $G$ and an integer $i\geq 1$, we denote by $\gamma_i(G)$ the $i$th term of the lower central series of $G$.

\begin{lemma}\label{le2}
Let $G$ be a group generated by a finite set of involutions. Then $\gamma_i(G)/\gamma_{i+1}(G)$ is a finite elementary abelian $2$-group.
\end{lemma}
\begin{proof}
It suffices to show that $\gamma_i(G)/\gamma_{i+1}(G)$ is generated by a finite set of involutions. Let $\{s_1,\ldots,s_d\}$ be a generating set of involutions for $G$. We argue by induction on $i$. Clearly, $\gamma_1(G)/\gamma_2(G)=G/[G,G]$ is an abelian group generated by $d$ involutions and hence the result holds for $i=1$. Assume that $i>1$. By the induction hypothesis, we have that $\gamma_{i-1}(G)$ is generated (modulo $\gamma_i(G)$) by $\{u_1,\ldots,u_k\}\subseteq\gamma_{i-1}(G)$ for some positive integer $k$. By the definition of the lower central series, $\gamma_i(G)/\gamma_{i+1}(G)$ is a central subgroup of $G/\gamma_{i+1}(G)$. Since $\gamma_i(G)=[\gamma_{i-1}(G),G]$, we see that  $\gamma_i(G)$ is generated (modulo $\gamma_{i+1}(G)$) by $[u_a,s_b]$, for $a\in \{1,\ldots,k\}$ and $b\in \{1,\ldots,d\}$. In particular, $\gamma_i(G)/\gamma_{i+1}(G)$ is finitely generated. Moreover, since $[u_a,s_b]\gamma_{i+1}(G)$ is a central element of $G/\gamma_{i+1}(G)$ and $s_b^2=1$, we have
\begin{eqnarray*}
[u_a,s_b]^2\gamma_{i+1}(G)&=&(u_a^{-1}s_bu_as_b)[u_a,s_b]\gamma_{i+1}(G)=u_a^{-1}s_b u_a[u_a,s_b]s_b\gamma_{i+1}(G)\\
&=&u_a^{-1}s_bu_a(u_a^{-1}s_bu_as_b)s_b\gamma_{i+1}(G)
=\gamma_{i+1}(G).
\end{eqnarray*}
Thus, $\gamma_{i}(G)/\gamma_{i+1}(G)$ is generated by a finite set of involutions. This completes the induction and the proof.
\end{proof}

Let $\{P_i\}_i$ denote the lower exponent-$2$ central series of $F_{d-1}$. In other words,

$$P_1=F_{d-1}\,\,\,\, \textrm{and}\,\,\,\, P_i=[P_{i-1},F_{d-1}]P_{i-1}^2, \,\,\, \, \textrm{for}\,\,\,i\geq 2.$$

It is clear from this definition that  $P_i$ is $\Sym(d)$-invariant.

\begin{lemma}\label{lee1}For every $i\geq 2$, we have $P_i=\gamma_i(W_d)$.
\end{lemma}
\begin{proof}
We argue by induction on $i$. By Lemma~\ref{le2}, $W_d/\gamma_2(W_d)$ is elementary abelian and hence $W_d^2\leq\gamma_2(W_d)$. Clearly, $F_{d-1}/F_{d-1}^2$ is elementary abelian. As $x_d$ has order $2$ and acts by inversion on the elements of $\{y_1,\ldots,y_{d-1}\}$, it follows that $W_d/F_{d-1}^2$ is abelian and hence $\gamma_2(W_d)\leq F_{d-1}^2$. By definition, $F_{d-1}^2\leq P_2\leq W_d^2$ and hence $W_d^2=\gamma_2(W_d)=F_{d-1}^2=P_2$ and the result holds for $i=2$.

Assume that $i\geq 3$. By the induction hypothesis, we have $P_i=[P_{i-1},F_{d-1}]P_{i-1}^2=[\gamma_{i-1}(W_d),F_{d-1}]\gamma_{i-1}(W_d)^2$. By Lemma~\ref{le2}, $\gamma_{i-1}(W_d)/\gamma_i(W_d)$ has exponent $2$ and hence $[\gamma_{i-1}(W_d),F_{d-1}]\gamma_{i-1}(W_d)^2\leq \gamma_i(W_d)$.

It remains to prove that $\gamma_i(W_d)\leq P_i$ or, by the induction hypothesis, that $[P_{i-1},W_d]\leq P_i$. By definition, $[P_{i-1},F_{d-1}]\leq P_i$ and, since $W_d=\langle F_{d-1},x_d\rangle$, it thus suffices to show that $[P_{i-1},x_d]\leq P_i$. As $P_{i-1}=[P_{i-2},F_{d-1}]P_{i-2}^2$, it suffices to show that $[P_{i-2}^2,x_d]$ and $[P_{i-2},F_{d-1},x_d]$ are both subgroups of $P_i$.

Let $g\in P_{i-2}$ and let $f\in F_{d-1}$. We will show that $[g^2,x_d]$ and $[g,f,x_d]$ are in $P_i$. Note that $P_{i-2}\leq\gamma_{i-2}(W_d)$ (this is trivial for $i=3$ and follows from the induction hypothesis for $i\geq 4$).  It follows that
$$[g,x_d]\in [P_{i-2},W_d]\leq[\gamma_{i-2}(W_d),W_d]=\gamma_{i-1}(W_d)=P_{i-1}.$$ Since $P_{i-1}^2$ and $[P_{i-1},F_{d-1}]$ are subgroups of $P_i$, it follows that $[g,x_d]^2$, $[x_d,g,f]$ and $[g,x_d,g]$ are in $P_i$. In particular, $[g^2,x_d]=[g,x_d]^g[g,x_d]=[g,x_d][g,x_d,g][g,x_d]\in P_i$. Finally, note that $[f,x_d]\in \gamma_2(W_d)=P_2$ and hence $[f,x_d,g]\in [P_2,P_{i-2}]\leq P_i$. Using the Hall-Witt identity, it follows that $[g,f,x_d]\in P_i$.
\end{proof}

We have thus shown that, starting at $i=2$, $\{P_i\}_i$ is both the lower central series of $W_d$ and the lower exponent-$2$ central series of $F_{d-1}$.

\subsection{A refined series}\label{seclow}
In this section, we describe an important refinement of the series $\{P_i\}_i$ and prove some results about this refined series. Detailed information  can be found in~\cite[Section~$20$]{Ne} whereas here, we simply give as much information as needed for our purposes.

For each positive integer $i\geq 1$ and for each $j\in \{0,\ldots,i\}$, let $M_{i,j}$ be the group generated by $P_{i+1}$ and by all elements of the form
\begin{equation}\label{gen}
[a_1,\ldots,a_{i-s}]^{2^{s}}
\end{equation}
where $a_1,\ldots,a_{i-s}\in \{y_1,\ldots,y_{d-1}\}$ and where $j\leq s\leq i-1$. It follows from~\cite[Lemma~$20.6$]{Ne} that
$$P_i=M_{i,0}\geq M_{i,1}\geq \cdots \geq M_{i,i-1}\geq M_{i,i}=P_{i+1}.$$
In particular, the series $(M_{i,j})_{i,j}$ is a refinement of the series $(P_i)_i$.

We now follow~\cite{BK}. Denote by $\Sigma$ the general linear group of degree $d-1$ over the field $\mathbb{F}_2$. Thus $\Sigma$ can be regarded as the group of $\mathbb{F}_2$-automorphisms of $P_1/P_2$. As $P_1$ is freely generated by $y_1,\ldots,y_{d-1}$, the action of each element of $\Sigma$ can be first extended to an automorphism of $P_1$ and then restricted to an automorphism of $P_i/P_{i+1}$, for each $i\geq 1$. Hence $\Sigma$ can be regarded as a group of automorphisms of each quotient $P_i/P_{i+1}$.  In this manner, we obtain an $\mathbb{F}_2\Sigma$-module structure on $P_i/P_{i+1}$ and, by restriction, also an $\mathbb{F}_2\Sym(d)$-module structure. We need to recall a few basic facts about this module. 

Koch~\cite{Koch} and Lazard~\cite{Lazard}  have both independently determined the submodule structure of the $p$-lower central series of a finitely generated free group for a prime $p\neq 2$. However, Lazard does not consider the case $p=2$ while Koch's work is unfortunately partially inaccurate (as brilliantly noticed by Bryant and Kov\'acs~\cite{BK}). We will thus refer to~\cite[Section~$3$]{BK} for information about the $\mathbb{F}_2\Sigma$-module $P_i/P_{i+1}$. 

From~\cite[Section~$3$, page~$421$]{BK}, we see that $M_{i,j}/P_{i+1}$ is an $\mathbb{F}_2\Sigma$-submodule of $P_i/P_{i+1}$, for each $i\geq 1$ and for each $j\in \{0,\ldots,i-2\}$. In particular, $M_{i,j}$ is $\Sym(d)$-invariant for every $j\neq i-i$. The case $j=i-1$ really must be excluded as $M_{i,i-1}/P_{i+1}$ is, in general, neither $\Sigma$- nor $\Sym(d)$-invariant. This can be explicitly checked when $d=3$ and $i=3$, for example.  

We now report some  properties of the $\mathbb{F}_2\Sigma$-module $M_{i,i-2}/P_{i+1}$. (To help the reader we use the notation from~\cite{BK}:  $M_{i,i-1}/P_{i+1}$ is denoted by $\Lambda_1$ and $M_{i,i-2}/P_{i+1}$ is denoted by $E$.) Let $\Lambda_2$ be the $\mathbb{F}_2$-subspace of $P_{i}/P_{i+1}$ spanned by $[y_r,y_s]^{2^{i-2}}$, for $r,s\in \{1,\ldots,d-1\}$. The definition of $M_{i,i-2}$ gives $$\frac{M_{i,i-2}}{P_{i+1}}=\Lambda_2+\frac{M_{i,i-1}}{P_{i+1}}.$$ 
In fact, this sum is a direct sum but this direct decomposition does not split as an $\mathbb{F}_2\Sigma$- or $\mathbb{F}_2\Sym(d)$-module. The subspace $\Lambda_2$ is actually an $\mathbb{F}_2\Sigma$-submodule and, more importantly, the quotient  $$\frac{M_{i,i-2}/{P_{i+1}}}{\Lambda_2}$$ is isomorphic to $P_1/P_2$ as an $\mathbb{F}_2\Sigma$-module. Since $\Sigma$ acts faithfully on $P_1/P_2$, it follows that $\Sigma$ (and hence $\Sym(d)$) acts faithfully on the section $(M_{i,i-2}/{P_{i+1}})/\Lambda_2$ of $P_i/P_{i+1}$ and therefore on $M_{i,i-2}/P_{i+1}$. Since these remarks are essential to our arguments, we collect them in the following proposition.

\begin{proposition}\label{lemma:208-}For every $i\geq 2$ and $j\in\{0,\ldots,i-2\}$, the group $M_{i,j}$ is $\Sym(d)$-invariant and $\Sym(d)$ acts faithfully on $M_{i,i-2}/M_{i,i}$.
\end{proposition}

We end this section with a remark (which can also be easily deduced from the previous discussion).

\begin{lemma}[{{\cite[page 203, lines~13 and~21]{Neumann}}}]\label{lemma:208--}
For every $i\geq 1$, $\rk(M_{i,i-1}/M_{i,i})=d-1$ and for every $i\geq 2$, $\rk(M_{i,i-2}/M_{i,i-1})=(d-1)(d-2)/2$. In particular, since $d\geq 3$, we have  $\rk(M_{i,i-2}/M_{i,i})<d!$.
\end{lemma}

\subsection{The main stuff}

We are now done with the preliminaries and are almost ready to prove the main results of this section. We briefly explain the general strategy. To find many subgroups that are normal and of index $2^m$ in $W_d$, we will find certain normal subgroups $H$ of $W_d$ of index less than $2^m$ such that $\rr(H)$ is ``rather large''. Now every subgroup $N$ of $H$ with $[H,W_d]H^2\leq N\leq H$ is normal in $W_d$. Moreover, since $H/([H,W_d]H^2)$ is an elementary abelian $2$-group, there will be many possible choices for $N$. Counting the subgroups $N$ that are of index $2^m$ in $W_d$ will give us the lower bound in Theorem~\ref{BoundsBounds}. A slightly more elaborate choice of $H$ will allow us to obtain the lower bounds in Theorems~\ref{Main1} and~\ref{Main2}.

\begin{lemma}\label{lemma208}
There exists a chain of normal subgroups $$W_d=H_0\geq H_1\geq \cdots \geq H_k\geq \cdots$$ of $W_d$ such that
\begin{description}
\item[(i)]for every $k\geq 0$, $|W_d:H_k|=2^k$;
\item[(ii)] $\rr(H_k)\geq (d-2)k/2+o(k)$, $(k\to\infty)$;
\item[(iii)]for every $k'\geq 0$, there exists $k\in \{k',k'+1,\ldots,k'+d!\}$ such that $H_{k}$ is $\Sym(d)$-invariant and $\Sym(d)$ acts faithfully on $H_k/([H_k,W_d]H_k^2)$.
\end{description}
\end{lemma}
\begin{proof}
We follow the proof of~\cite[Lemma~$20.8$]{Ne}, with some modifications. For $i\geq 3$ and $j\in \{0,1,\ldots,i-3\}$, let
$$M_{i,j}=V_{i,j,0}\geq \cdots \geq V_{i,j,t_{i,j}}=M_{i,j+1}$$
be a composition series for the $\mathbb{F}_2\Sym(d)$-module $M_{i,j}/M_{i,j+1}$. In particular, for every $x\in \{1,\ldots,t_{i,j}\}$, the $\mathbb{F}_2\Sym(d)$-module $V_{i,j,x-1}/V_{i,j,x}$ is irreducible and thus has dimension at most $|\Sym(d)|=d!$. Let
$$W_d=H_0\geq H_1\geq H_2\geq \cdots$$
be a series for $W_d$ which 
\begin{enumerate}
\item includes $M_{i,j}$ for every $i\geq 1$ and every $j\in\{0,\ldots,i\}$, 
\item includes $V_{i,j,x}$ for every $i\geq 3$, every $j\in \{0,1,\ldots,i-3\}$ and every $x\in \{0,\ldots,t_{i,j}\}$, and
\item such that $|H_k:H_{k+1}|=2$ for every $k\geq 0$. 
\end{enumerate}
This shows~\textbf{(i)}. We now prove~\textbf{(ii)}. It follows from the proof of~\cite[Lemma~$20.8$]{Ne} that $\rk(H_k/([H_k,F_{d-1}]H_k^2))\geq (d-2)k+o(k)$. Let $U=H_k/([H_k,F_{d-1}]H_k^2)$ and observe that $H_k/([H_k,W_d]H_k^2)$ is isomorphic to $U/[U,W_d]$. As $W_d=\langle F_{d-1},x_d\rangle$ and as $F_{d-1}$ centralises $U$, we have $[U,W_d]=[U,x_d]$. Consider the map $\pi:U\to [U,x_d]$ defined by $\pi(u)=[u,x_d]$.  Now, since $U$ is abelian, we have that $\pi$ is a surjective homomorphism with kernel $\cen{U}{x_d}=\{u\in U \mid u^{x_d}=u\}$. Since $x_d$ has order $2$, it is easy to check that $[U,x_d]\leq \cen{U}{x_d}$. By the first isomorphism theorem, $|U|=|\ker(\pi)||\im(\pi)|=|\cen{U}{x_d}||[U,x_d]|\leq |\cen{U}{x_d}|^2$. It follows

$$\left|\frac{H_k}{[H_k,W_d]H_k^2}\right|=\left|\frac{U}{[U,x_d]}\right|=|\cen{U}{x_d}|\geq |U|^{1/2}.$$
Recall that $\rk(U)=\rk(H_k/([H_k,F_{d-1}]H_k^2))\geq (d-2)k+o(k)$ and~\textbf{(ii)} follows. It only remains to prove~\textbf{(iii)}. Given $k'$, choose $k$ in the following way.

\begin{enumerate}
\item If $H_{k'}> M_{1,1}$, then choose $k$ such that $H_k=M_{1,1}=M_{2,0}$.
\item If $M_{i,i-2}\geq H_{k'}\geq M_{i,i}$ for some $i\geq 2$, then choose $k$ such that $H_k=M_{i,i}=M_{i+1,0}$.
\item Otherwise, choose $k$ such that $H_k=V_{i,j,x}$ for some $i\geq 3$, $j\in \{0,1,\ldots,i-3\}$, and $x\in \{0,\ldots,t_{i,j}\}$.
\end{enumerate}

We now show that we can choose $k$ as above and such that $k\in \{k',k'+1,\ldots,k'+d!\}$. By Lemma~\ref{lemma:208--}, we have $\rk(W_d/M_{1,1})=d<d!$ and $\rk(M_{i,i-2}/M_{i,i})<d!$, hence our claim is certainly true in cases $(1)$ and $(2)$. 

If $k'$ is such that we are in neither case $(1)$ nor case $(2)$, then certainly $M_{3,0}>H_{k'}$ and, since we are not in case $(2)$, we have $M_{i,j}\geq H_{k'}\geq M_{i,j+1}$ for some $i\geq	3$ and $j\in \{0,1,\ldots,i-3\}$. It follows that $V_{i,j,x-1}\geq H_{k'}\geq V_{i,j,x}$ for some $i\geq 3$, $j\in \{0,1,\ldots,i-3\}$, and $x\in \{1,\ldots,t_{i,j}\}$. As noted earlier, $\rk(V_{i,j,x-1}/V_{i,j,x})\leq d!$ and hence choosing $k$ such that $H_k=V_{i,j,x}$ satisfies our requirement. This completes the proof of the first part of~\textbf{(iii)}.

On the other hand, it is clear from the construction together with Proposition~\ref{lemma:208-} that $H_k$ is $\Sym(d)$-invariant, completing the proof of the second part of~\textbf{(iii)}.

Now, note that by our choice of $H_k$, we have that there exists an $i\geq 2$ such that $M_{i,0}\geq H_k\geq M_{i,i-2}$. Recall that $M_{i,0}=P_i$ and hence
$$[H_k,W_d]H_k^2\leq [P_i,W_d]P_i^2.$$
By Lemma~\ref{lee1}, we have $[P_i,W_d]P_i^2=P_{i+1}$. In particular, $[H_k,W_d]H_k^2 \leq P_{i+1}=M_{i,i}$. This shows that $M_{i,i-2}/M_{i,i}$ is a section of $H_k/([H_k,W_d]H_k^2)$. It follows from Proposition~\ref{lemma:208-} that $\Sym(d)$ acts faithfully on  $M_{i,i-2}/M_{i,i}$ and we have proved~(\textbf{iii}).

\end{proof}

We are now ready to prove our main theorem. During the proof, we will appeal to Lemma~\ref{apeman} the proof of which is rather technical and delayed until Section~\ref{sec:technical}. Let $\nor{G}{H}$ denote the normaliser of the group $H$ in $G$ and recall that $g_d(m)$ denotes the number of isomorphism classes of groups of order $2^m$ which admit a generating set consisting of $d$ involutions.

\begin{theorem}\label{thm2}
Let $d\geq 3$. Then $g_d(m)\geq 2^{\frac{(d-2)^2}{8d}m^2+o(m^2)}$. Moreover, the number of isomorphism classes of Cayley graphs $\Gamma=\Cay(G,S)$ of order $2^m$ with $S$ consisting of $d$ involutions and with $\nor{\Aut(\Gamma)}{G}=G$ is also at least $2^{\frac{(d-2)^2}{8d}m^2+o(m^2)}$.
\end{theorem}
\begin{proof}
We follow the proof of~\cite[Theorem~$20.4$]{Ne}. Let $m$ be a positive integer and let $k'=\lceil \frac{d+2}{2d}m \rceil$. By Lemma~\ref{lemma208}, there exists an integer $k\in \{k',k'+1,\ldots,k'+d!\}$ and a normal subgroup $H_k$ of $W_d$ of index $2^k$ such that $H_{k}$ is $\Sym(d)$-invariant, $\Sym(d)$ acts faithfully on $H_k/([H_k,W_d]H_k^2)$ and
\begin{equation}\label{eq204b}
\rr(H_k)\geq (d-2)k/2+o(k), ~(k\to\infty).
\end{equation}

Let $s=m-k$. Clearly, $k=\frac{d+2}{2d}m+o(m)$ and hence $s=\frac{d-2}{2d}m+o(m)$. Write $K=[H_k,W_d]H_k^2$.
Let
$$
\NM=\{N\leq H_k \mid K\leq N,\, |H_{k}:N|=2^s,\, N^\alpha\neq N \textrm{ for every } 1\neq\alpha\in \Sym(d)\}.
$$

By definition, $|H_k:K|=2^{\rr(H_k)}$. Since $H_k/K$ is an elementary abelian $2$-group, it follows from Lemma~\ref{apeman} that $|\NM|\geq 2^{s(\rr(H_{k})-s)+o(s)}$. Using Eq.~\eqref{eq204b}, we obtain

\begin{eqnarray*}
s(\rr(H_k)-s)+o(s)&\geq& \left(\frac{d-2}{2d}m+o(m)\right)\left(\frac{(d-2)(d+2)}{4d}m-\frac{d-2}{2d}m+o(m)\right)+o(m)\\
&=&\frac{(d-2)^2}{8d}m^2+o(m^2),
\end{eqnarray*}
and hence $|\NM|\geq 2^{\frac{(d-2)^2}{8d}m^2+o(m^2)}$. Since $[H_k,W_d]\leq K$, it follows that $H_k/K$ is central in $W_d/K$ and hence every group in $\NM$ is normal in $W_d$.  Moreover, every group in $\NM$ has index $2^m$ in $W_d$. Therefore, each $N$ in $\NM$ gives rise to a quotient group $W_d/N$ of order $2^m$ generated by $d$ involutions $\{x_1N,\ldots,x_dN\}$.

Let $G$ be a group of order $2^m$ generated by $d$ involutions. The number of normal subgroups $N$ of $W_d$ such that $G\cong W_d/N$ is equal to the number of surjective homomorphisms from $W_d$ to $G$. Since $W_d$ is $d$-generated, the number of such homomorphisms is at most $|G|^d=2^{md}$.
It follows that $g_d(m)\geq|\NM|/2^{md}=2^{\frac{(d-2)^2}{8d}m^2+o(m^2)}$, concluding the first part of the theorem.

Let $N\in\NM$, let $G=W_d/N$ and define $\Gamma=\Cay(G,\{x_1N,\ldots,x_dN\})$. Let $A=\Aut(\Gamma)$ and let $X=\nor A G$. Clearly, $X=G\rtimes X_1$, where $X_1$ denotes the stabiliser in $X$ of the vertex corresponding to the identity in $G$. We show that $X_1$ is trivial and hence that $X=G$. Let $g\in X_1$. The element $g$ acts as an automorphism of $G$ permuting the connection set $\{x_1N,\ldots,x_dN\}$. As the only relations in $W_d$ are $x_1^2=1,\ldots,x_d^2=1$, the element $g$ lifts to an automorphism $\alpha_g$ of $W_d$ lying in $\Sym(d)$ and with $N^{\alpha_g}=N$. From the definition of $\NM$, it follows that $\alpha_g=1$ and hence $g=1$. Thus $\nor {\Aut(\Gamma)}{G}=G$.

For $i\in\{1,2\}$, let $N_i\in\NM$, let $G_i=W_d/N_i$, define $\Gamma_i=\Cay(G_i,\{x_1N_i,\ldots,x_dN_i\})$ and assume that $\Gamma_1\cong \Gamma_2$. In particular, $\Aut(\Gamma_1)\cong \Aut(\Gamma_2)$. Since $\nor {\Aut(\Gamma_i)}{G_i}=G_i$ and $G_i$ is a $2$-group, it follows that $G_i$ is a Sylow $2$-subgroup of $\Aut(\Gamma_i)$ but then $G_1\cong G_2$. As we have seen, the number of different isomorphism classes for $G_i$ is at least $2^{\frac{(d-2)^2}{8d}m^2+o(m^2)}$ hence so is the number of different isomorphism classes for $\Gamma_i$.
\end{proof}

\begin{proof}[Proof of the lower bounds in Theorems~\ref{Main1} and~\ref{BoundsBounds}]
The lower bound in Theorem~\ref{BoundsBounds} follows immediately from Theorem~\ref{thm2}, whilst the lower bound in Theorem~\ref{Main1} follows from Theorem~\ref{thm2} applied with $m=\lfloor\log_2n\rfloor$.
\end{proof}

It remains only to prove the lower bound in Theorem~\ref{Main2}. To do this, we combine Theorem~\ref{thm2} with a result of Li.

\begin{corollary}\label{cor1}
There are at least $2^{\frac{m^2}{24}+o(m^2)}$ $3$-valent GRRs of order $2^m$.
\end{corollary}
\begin{proof}
By~\cite[Theorem]{Li}, we see that if $\Gamma=\Cay(G,T)$ is a $3$-valent Cayley graph on a $2$-group and $\nor {\Aut(\Gamma)}{G}=G$, then $\Aut(\Gamma)=G$, that is, $\Gamma$ is a GRR. Now the result follows from Theorem~\ref{thm2} applied with $d=3$.
\end{proof}

\begin{proof}[Proof of the lower bound in Theorem~\ref{Main2}]
It follows from Corollary~\ref{cor1} applied with $m=\lfloor \log_2n\rfloor$.
\end{proof}

\subsection{Technicalities : Lemma~\ref{apeman}}\label{sec:technical}

Given $r\geq 1$, $q\geq 2$ and $s$ with $0\leq s\leq r$, define $${r\choose s}_q=\frac{q^r-1}{q^s-1}\frac{q^{r-1}-1}{q^{s-1}-1}\cdots \frac{q^{r-s+1}-1}{q-1}.$$
Observe that, when $q$ is  a prime power, ${r\choose s}_q$ is the number of $s$- or $(r-s)$-dimensional subspaces of an $r$-dimensional vector space over the finite field $\mathbb{F}_q$ of order $q$. Moreover, as  $\frac{2^{r-i}-1}{2^{s-i}-1}\geq 2^{r-s}$, we have ${r\choose s}_2\geq 2^{s(r-s)}$.

\begin{lemma}\label{apeman}Let $r$ and $s$ be integers with $r>s>0$. Let $V$ be an elementary abelian $2$-group of order $2^r$, let $T\leq \GL(V)$ with $|T|=O(1)$ and let $\mathcal{N}_{r,s}=\{W\leq V\mid |V:W|=2^s,\,W^\alpha\neq W\,\textrm{for every }\alpha\in T\setminus\{1\}\}$. Then $|\mathcal{N}_{r,s}|\geq 2^{s(r-s)+o(r^2)}$ as $\min\{s,r-s\}\to \infty$.
\end{lemma}
\begin{proof}
Let $\alpha\in T\setminus\{1\}$, let $\Fix(\alpha)=\{W\leq V\mid W^\alpha=W,\,|V:W|=2^s\}$ and let $C=\{v\in V\mid v^\alpha=v\}$. We start by obtaining an upper bound on $|\Fix(\alpha)|$.  Clearly, for every $n\in \mathbb{Z}$, we have $|\Fix(\alpha)|\leq |\Fix(\alpha^n)|$ and hence, by replacing $\alpha$ by a suitable power, we may assume that $\alpha$ has prime order $p$. We now distinguish two cases, depending on whether $p=2$ or $p>3$. In both cases, we obtain an exact formula for $|\Fix(\alpha)|$.

Suppose first that $p=2$.  Note that, as $|\alpha|=2$, every Jordan block of $\alpha$ has size $1$ or $2$ and hence $[V,\alpha]\leq C$. Choose an element $W$ of $\Fix(\alpha)$ and write $U=W+C$. Since $\alpha$ centralises $C$, we have $[U,\alpha]=[W+C,\alpha]=[W,\alpha]$ and hence $[U,\alpha]\leq W\cap C$. It follows that any triple consisting of:
\begin{enumerate}
\item a subspace $U/C$ of $V/C$,
\item a subspace $(W\cap C)/[U,\alpha]$ of $C/[U,\alpha]$, and
\item a complement $W/(W\cap C)$ of $C/(W\cap C)$ in $U/(W\cap C)$
\end{enumerate}
 determines a unique element $W$ of $\Fix(\alpha)$ and conversely. It thus suffices to count the number of possible choices for each of these subspaces. Let $t$ be the number of Jordan blocks of $\alpha$ of size $2$ and let $x=\dim(U/C)$. Then $\dim(V/C)=t$ and hence the number of possible choices for $U/C$ is ${t\choose x}_2$. Note that $\dim([U,\alpha])=\dim([W,\alpha])=\dim(W/(W\cap C))=\dim(U/C)=x$. Moreover $\dim(V/C)=r-t$ and hence $\dim(C/[U,\alpha])=r-t-x$. Next, observe that $\dim(C/(W\cap C))=\dim(U/W)=\dim(V/W)-\dim(V/U)=s-(t-x)=s-t+x$ and hence $\dim((W\cap C)/[U,\alpha])=r-t-x-(s-t+x)=r-s-2x$.

It follows that $t-s\leq x\leq (r-s)/2$ and that the number of possible choices for $(W\cap C)/[U,\alpha]$ is ${r-t-x\choose r-s-2x}_2$. Finally, the number of choices for a complement $W/(W\cap C)$ of $C/(W\cap C)$ is $|C/(W\cap C)|^{\dim(W/(W\cap C))}=2^{(s-t+x)x}$. Summing up,
\begin{equation}\label{eqeq1}
|\mathrm{Fix}(\alpha)|=
\sum_{x\geq \max\{0,t-s\}}^{\min\{t,\lfloor\frac{r- s}{2}\rfloor\}}
{t\choose x}_2{r-t-x\choose r-s-2x}_22^{(s-t+x)x}.
\end{equation}
A rather tedious computation with Eq.~\eqref{eqeq1} shows that the right-hand side attains its maximum when $t=1$, in other words, when $\alpha$ has only one Jordan block of size 2.

Suppose now that $p>2$. Let $\ell$ be the smallest positive integer such that $p$ divides $2^\ell-1$.  Since $|\alpha|$ is coprime to the characteristic of $V$, we have $V=[V,\alpha]+C$ with $[V,\alpha]\cap C=0$.  Let $t$ denote the number of Jordan blocks of $\alpha$ on $[V,\alpha]$. From our choice of $\ell$ and from Schur's lemma, the action of $\alpha$ on $[V,\alpha]$ is  conjugate to a scalar matrix in $\GL_{t}(2^\ell)$ corresponding to a field generator of $\mathbb{F}_{2^\ell}$. In particular, the $\mathbb{F}_2\langle\alpha\rangle$-invariant subspaces of $[V,\alpha]$ are in one-to-one correspondence with the $\mathbb{F}_{2^\ell}$-subspaces of the $t$-dimensional vector space $\mathbb{F}_{2^\ell}^t$. 

Choose an element $W$ of $\Fix(\alpha)$.  As $p\neq 2$, the coprime action of $\alpha$ on $W$ gives $W=[W,\alpha]+(W\cap C)$. From the previous paragraph, any pair consisting of:
\begin{enumerate}
\item an $\mathbb{F}_{2^\ell}$-subspace of $\mathbb{F}_{2^\ell}^t$ corresponding to $[W,\alpha]$, and 
\item an $\mathbb{F}_2$-subspace $W\cap C$ of $C$
\end{enumerate}
determines a unique element $W$ of $\Fix(\alpha)$ and conversely. It thus suffices to count the number of possible choices for each of these subspaces. The number of choices for an $\mathbb{F}_{2^\ell}$-subspace of $\mathbb{F}_{2^\ell}^t$ of dimension $x$ is ${t\choose x}_{2^\ell}$. Choosing such a subspace of dimension $x$ corresponds to choosing $[W,\alpha]$ of dimension $x\ell$ and hence $\dim(W\cap C)=r-s-x\ell$. Note that $\dim (C)=r-t\ell$ and hence the number of choices of an $\mathbb{F}_2$-subspace of $C$ of dimension $r-s-x\ell$ is ${r-t\ell\choose r-s-x\ell}_2$. Moreover, $0\leq r-s-x\ell\leq r-t\ell$, that is $t-s/\ell\leq x\leq (r-s)/\ell$.  Summing up,
\begin{equation}\label{eqeq2}
|\mathrm{Fix}(\alpha)|=
\sum_{x\geq \max\{0,t-s/\ell\}}^{\min\{t,(r-s)/\ell\}}
{t\choose x}_{2^\ell}
{r-t\ell\choose r-s-x\ell}_2.
\end{equation}
Another rather tedious computation with Eq.~\eqref{eqeq2} shows that the right-hand side attains its maximum when $t=1$ and $\ell=2$. Moreover, this maximum is less than the maximum of the right-hand side of Eq.~\eqref{eqeq1}. This shows that the maximum of $|\Fix(\alpha)|$ as $\alpha$ runs over the  elements of prime order of $\mathrm{GL} (V)$ is achieved when $\alpha$ is  an involution having only one Jordan block of size $2$, in other words, when $\alpha$ is a transvection. This is the case $t=1$ in Eq.~\eqref{eqeq1} and hence
\begin{equation}\label{eqeq3}
|\Fix(\alpha)|\leq {r-1\choose s-1}_2+{r-2\choose s}2^s, \textrm{ for every $\alpha\in T\setminus\{1\}$}.
\end{equation}
Denote by $f_{r,s}$ the right-hand side of Eq.~\eqref{eqeq3}. We have

\begin{eqnarray*}
\frac{f_{r,s}}{\left(\begin{array}{c}r\\s\end{array}\right)_2}&=&\left(\frac{2^s-1}{2^r-1}+\frac{(2^{r-s}-1)(2^{r-s-1}-1)}{(2^{r}-1)(2^{r-1}-1)}2^{s}\right)\\
&\leq&\left(\frac{1}{2^{r-s}}+\frac{1}{2^{s}}\right)\leq 2\cdot\frac{1}{2^{\min\{r-s,s\}}}=2^{1-\min\{r-s,s\}}.
\end{eqnarray*}
It follows that
\begin{eqnarray*}
\frac{|\mathcal{N}_{r,s}|}{\left(\begin{array}{c}r\\s\end{array}\right)_2}&\geq& \frac{\left(\begin{array}{c}r\\s\end{array}\right)_2-(|T|-1)f_{r,s}}{\left(\begin{array}{c}r\\s\end{array}\right)_2}=1-(|T|-1)\frac{f_{r,s}}{\left(\begin{array}{c}r\\s\end{array}\right)_2}\\
&\geq&1-(|T|-1)2^{1-\min\{r-s,s\}}.
\end{eqnarray*}
Recall that $T=O(1)$ and ${r\choose s}_2\geq 2^{s(r-s)}$ and the lemma follows.
\end{proof}

\section{The proof of Theorem~\ref{theo:5arc}}\label{sec:5AT}
The study of $3$-valent $s$-arc-transitive graphs was initiated in 1947 by Tutte~\cite{tutte1} who proved that $s\leq 5$. Tutte also constructed the first example of a $3$-valent $5$-arc-transitive graph, a graph of order $30$ known as the Tutte-Coxeter graph or Tutte eight-cage.

According to Biggs, the first example on an infinite family of $3$-valent $5$-arc-transitive graphs was given by Conway~\cite[p.130]{Biggsbook}. Since then, many other constructions have been found but they are usually rather ``sparse" in the sense that they do not yield many graphs up to a given order. It was unclear whether this was due to the actual sparseness of this family of graphs or simply to our lack of understanding. Theorem~\ref{theo:5arc} settles this question by showing that $3$-valent $5$-arc-transitive graphs are rather frequent, at least in the asymptotical sense.

\begin{proof}[Proof of Theorem~\ref{theo:5arc}]
Let $\Delta$ be the Tutte $8$-cage and let $H=\mathrm{Aut}(\Delta)$. Choose a vertex $v$ of $\Delta$ and let $w$ be a neighbour of $v$. Define $A=H_v$, $B=H_{\{v,w\}}$, $C=H_{vw}$ and $G=A\ast_CB $. Let $\pi:G\to H$ be the natural projection  and let $N$  be the kernel of $\pi$. Observe that, by construction, $\pi_{|A}$ and $\pi_{|B}$ are injective and hence $N$ intersects $A$ and $B$ trivially. Therefore, by~\cite[Proposition~I.$5.4$]{Wood}, $N$ is a free group.  The group $G$ has a natural action as a transitive group of automorphisms of the infinite $3$-valent tree $T$. As $N\unlhd G$ and $|G:N|=|H|<\infty$, we see that $N$ has a finite number of orbits on $\V(T)$. Since $T$ is $3$-valent, this forces $N$ to have  rank at least $2$.

As $|C|=|H_{vw}|=16$, we see that $C$ is $4$-generated. Moreover, as $C$ is a maximal subgroup of $A$ and of $B$, it follows that $G$ is $6$-generated. Define
$$\mathcal{N}_n=\{M\unlhd G\mid M\leq N,\,|G:M|\leq 48n\}.$$
Let $p$ be a prime coprime to $|G:N|$ and observe that, since $N$ is a free group of rank at least $2$, the pro-$p$-completion of $N$ is a free pro-$p$-group of rank at least $2$. Thus~\cite[Theorem~$1$]{MP} yields that there exist two positive constants $a'$ and $b'$ such that  $|\mathcal{N}_n|\geq n^{a'\log n}-b'$. (We thank A.~Mann for pointing out this reference to us.)

Now we recall the definition of coset graph. For a group $G$, a subgroup $A$ and an element $b \in G$, the coset graph $\mathrm{Cos}(G, A, b)$ is the graph with vertex set the set of right cosets $G/A =\{ Ag\mid g\in G\}$ and edge set $\{\{ Ag,Abg\}\mid g\in G\}$.

Let $M\in \mathcal{N}_n$, let $b\in B\setminus C$ and define $\Gamma=\mathrm{Cos}(G/M,AM/M,bM)$. Since $M\leq N$, the graph $\Gamma$ is a regular cover of $\Delta$. Moreover, since $G/M$ acts as a group of automorphisms of $\Gamma$ with vertex-stabilisers isomorphic to $AM/M\cong A/(A\cap M)\cong A=H_v$, it follows that $\Gamma$ is $5$-arc-transitive. As $|H_v|=48$, it follows that $$|\V(\Gamma)|=\frac{|G:M|}{|AM:M|}\leq \frac{48n}{48}=n$$ and that $|G/M|=48|\V(\Gamma)|$. As the vertex-stabiliser of a  $3$-valent $5$-arc-transitive graph has order $48$~\cite{tutte1,tutte2}, we have $|\mathrm{Aut}(\Gamma)|=48|\V(\Gamma)|$ and hence  $\mathrm{Aut}(\Gamma)=G/M$. Summing up, we have shown that every element $M$ of $\mathcal{N}_n$ determines a $3$-valent $5$-arc-transitive graph with automorphism group $G/M$.

Finally, for $i\in \{1,2\}$, let $M_i\in \mathcal{N}_n$, let $b_i\in B\setminus C$ and define the graph $\Gamma_i=\mathrm{Cos}(G/M_i,AM_i/M_i,b_iM_i)$.  If $\Gamma_1\cong \Gamma_2$, then $\mathrm{Aut}(\Gamma_1)\cong \mathrm{Aut}(\Gamma_2)$ and hence $G/M_1\cong G/M_2$. The number of normal subgroups $M$ of $G$ with $G/M\cong G/M_1$ is equal to the number of surjective homomorphisms from $G$ to $G/M_1$. Since $G$ is $6$-generated, the number of such homomorphisms is at most $|G/M_1|^6\leq(48n)^6$. We conclude that the number of isomorphism classes of $3$-valent $5$-arc-transitive graphs of order at most $n$ is at least $|\mathcal{N}_n|/(48n)^6\geq (n^{a'\log n}-b')/(48n)^6\geq n^{a\log n}-b$, for some $a,b>0$.
\end{proof}

\noindent\textbf{Remark.} The proof of Theorem~\ref{theo:5arc} only relies on a few properties of $3$-valent $5$-arc-transitive graphs and hence the hypothesis of Theorem~\ref{theo:5arc} could be considerably weakened with only a little more effort. We chose not to do it here to avoid too large a digression but we plan to return to this question in future work and decided to leave Theorem~\ref{theo:5arc} as a teaser.

\end{document}